\documentclass[a4paper,12pt]{amsart}
\usepackage{graphicx,amssymb}
\usepackage{txfonts}
\usepackage{mathrsfs}
\usepackage{amssymb,amsmath,amsthm,color}
\usepackage{graphicx,mcite}
\usepackage{hyperref}
\usepackage{url}
\usepackage{setspace}

\newtheorem{theorem}{Theorem}[section]

\newtheorem{lemma}[theorem]{Lemma}
\newtheorem{definition}[theorem]{Definition}

\newtheorem{remark}[theorem]{Remark}
\newtheorem{corollary}[theorem]{Corollary}

\def\be#1 {\begin{equation} \label{#1}}
\newcommand{\ee}{\end{equation}}

\def\sqw{\hbox{\rlap{\leavevmode\raise.3ex\hbox{$\sqcap$}}$%
\sqcup$}}
\def\findem{\ifmmode\sqw\else{\ifhmode\unskip\fi\nobreak\hfil
\penalty50\hskip1em\null\nobreak\hfil\sqw
\parfillskip=0pt\finalhyphendemerits=0\endgraf}\fi}

\newcommand{\R}{{\mathbb {R}}}

\author{K. Jotsaroop, Saurabh Shrivastava, Kalachand Shuin}
\address{
}
\email{}
\address{
Department of Mathematics\\
Indian Institute of Science Education and Research Bhopal\\
Bhopal-462066, India}
\email{saurabhk@iiserb.ac.in} 
\email{kalachand16@iiserb.ac.in}

\address{Department of Mathematics, Indian Institute of Science Education and Research, Mohali  India}
\email{jotsaroop@iisermohali.ac.in}
\keywords{Bilinear maximal function, Bilinear Muckenhoupt weights, Extrapolation theory, Bilinear Bochner-Riesz operator}
\subjclass{42B20, 42B25}

\begin{document}

\title[Bilinear Bochner-Riesz means at the critical index]{Weighted estimates for Bilinear Bochner-Riesz means at the critical index}

\begin{abstract}
In this paper we establish weighted estimates for the bilinear Bochner-Riesz operator $\mathcal B^{\alpha}$ at the critical index $\alpha=n-\frac{1}{2}$ with respect to bilinear weights. 
\end{abstract}

\maketitle
\section{Introduction}
\subsection{Bochner-Riesz means}
The study of linear Bochner-Riesz means is a vast subject and originates from the classical problem of summability of Fourier series. The $n-$dimensional Bochner-Riesz operator of order $\alpha\geq 0$ acting on a Schwarz class function $f\in S(\R^n)$ is defined by
\begin{eqnarray*}
\mathcal {S}^{\alpha}(f)(x):=\int_{\mathbb{R}^{n}}(1-|\xi|^{2})^{\alpha}_{+}\hat{f}(\xi)e^{2\pi \iota x\cdot\xi}d\xi,~~x\in \R^n, \\
\end{eqnarray*}
where $x\cdot y$ denotes the standard inner product in $\R^n$, $r_+=r$ if $r> 0$ and $r_+=0$ if $r\leq 0$. 
The study of $L^p$ boundedness properties of the operator $\mathcal {S}^{\alpha}$ has been a central theme in Harmonic analysis. The Bochner-Riesz conjecture is one of the outstanding open problems in the subject. The conjecture concerns finding the best possible range of exponents $p$, for a given index  $\alpha$, for which the operator $\mathcal {S}^{\alpha}$ is bounded on  $L^p(\R^n)$. The problem has been studied by many mathematicians and is well understood in dimension $n=1,2$. However, it remains open to date in dimension $n\geq 3$. We refer the reader to \cite{B1,B2,F1,F2, Lee2, Lee1} and references therein for specific details. 

Note that the convolution kernel of the Bochner-Riesz operator $\mathcal {S}^{\alpha}$ is given by 
\begin{eqnarray*}
K_{\alpha}(x)=c_{\frac{n}{2}+\alpha} \frac{ J_{\alpha+\frac{n}{2}} (2\pi |x|)} {|x|^{\alpha+\frac{n}{2}}},~x\in \R^n,
\end{eqnarray*}
where $J_{\alpha+\frac{n}{2}}$ denotes the standard Bessel function of order $\alpha+\frac{n}{2}$. Using the properties of Bessel functions, it is easily verified that for $\alpha>\frac{n-1}{2}$, the kernel $K_{\alpha}$ is an integrable function. As a consequence of this, the $L^p$ boundedness, $1\leq p\leq \infty,$ of the operator $\mathcal {S}^{\alpha}$ follows immediately. However, for $0\leq \alpha\leq \frac{n-1}{2}$ the problem is known to be more difficult. The case $\alpha=0$ is referred to as the ball multiplier problem. In~\cite{F2}, C. Fefferman proved that the ball multiplier operator is unbounded on $L^p(\R^n), n\geq 2,$ for $p\neq 2.$ 

The index $\alpha= \frac{n-1}{2}$ is commonly referred to as the critical index for the Bochner-Riesz problem. The Bochner-Riesz operator $\mathcal {S}^{\frac{n-1}{2}}$ has a close connection with rough singular integral operators. The operator $\mathcal {S}^{\frac{n-1}{2}}$ is bounded on $L^p(\R^n)$ for $1<p<\infty$ and is of weak-type $(1,1)$, see~\cite{MC,Di,Seeger,St1,AV} for details. The literature is vast and here we do not attempt to provide an exhaustive account of the subject.  We only discuss the results which directly concern the current paper. 

Next, we move on to discuss the weighted boundedness of the Bochner-Riesz operator $\mathcal {S}^{\frac{n-1}{2}}$. In order to describe the results, we need to recall the notion of Muckenhoupt weights. 

For a given $1<p<\infty$, the Muckhenhoupt class of weights $A_p$ consists of all non-negative locally integrable functions $w$ satisfying
$$
[w]_{A_p}:=\sup_B \bigg(\fint_B w \bigg)\bigg(\fint_B w^{1-p'} \bigg)^{p-1} <\infty,
$$
where the supremum ranges over all balls $B$ in $\R^n$. Here $p'=\frac{p}{p-1}$ denotes the conjugate index to $p$ and $\fint_B f:=\frac{1}{|B|}\int_B |f(y)|dy$ is the average of $f$ over ball $B.$ 

For $p=1,$ the class $A_1$ consists of all weights $w$ such that 
$$
[w]_{A_1}:={\it{ess}sup}~ \frac{M (w)}{w}<\infty.
$$
Here $M$ denotes the classcial Hardy-Littlewood maximal operator defined by 
$$
M(f)(x):=\sup_{B:x\in B} \fint_B f.
$$
The constant $[w]_{A_p}, 1\leq p<\infty,$ is referred to as the $A_p$ characteristic constant of the weight $w$.
 
In~\cite{XQ},~X. Shi and Q. Sun proved  weighted $L^p$ estimates for the operator $\mathcal {S}^{\frac{n-1}{2}}$ for $1<p<\infty$ with respect to $A_p$ weights. In fact, they obtained weighted estimates for the maximal Bochner-Riesz operator at the critical index. Later, in~\cite{AV}, A. Vargas proved weighted weak-type estimates for the operator $\mathcal {S}^{\frac{n-1}{2}}$ at the end-point $p=1$ with respect to $A_1$ weights. There have been recent developments on the problem in connections with the sparse domination principle and sharp weighted bounds. We would like to refer the interested reader to ~\cite{Di} for more details.  

In recent times, there has been some progress on the Bochner-Riesz problem in the bilinear setting. Motivated by these works,  in this article we address the question of weighted boundedness of the bilinear Bochner-Riesz operator at the critical index with respect to bilinear weights. In the next section we discuss the bilinear Bochner-Riesz operators in detail. 

\subsection{Bilinear Bochner-Riesz means} 

The bilinear Bochner-Riesz operator of order  $\alpha\geq 0$ in $\R^n,~n\geq 1$, is the bilinear multiplier operator defined by  
\begin{eqnarray*}
\mathcal {B}^{\alpha}(f,g)(x)&:=&\int_{\mathbb{R}^{n}}\int_{\mathbb{R}^{n}}(1-\vert \xi\vert^{2}-\vert \eta\vert^{2})^{\alpha}_{+}\hat{f}(\xi)\hat{g}(\eta)e^{2\pi \iota x\cdot(\xi+\eta)}d\xi d\eta\\
& := &\int_{\mathbb{R}^{n}}\int_{\mathbb{R}^{n}}K_{\alpha}(y,z)f(x-y)g(x-z) dydz,
\end{eqnarray*} 
where as in the previous section 
$$K_{\alpha}(y,z)=\int_{\mathbb{R}^{n}}\int_{\mathbb{R}^{n}}(1-\vert \xi\vert^{2}-\vert \eta\vert^{2})^{\alpha}_{+}e^{2\pi \iota (y\cdot\xi+z\cdot\eta)}d\xi d\eta=c_{n+\alpha} \frac{ J_{\alpha+n} (2\pi |(y,z)|)} {|(y,z)|^{\alpha+n}},~y,z\in \R^n.$$
The bilinear Bochner-Riesz operators are natural analogues of the classical Bochner-Riesz operators. They have their origin in the study of double Fourier series and in general, in the study of non-linear PDE's, see~\cite{tao} for more details. It is natural to investigate the boundedness properties of the bilinear Bochner-Riesz operators from $L^{p_1}(\R^n)\times L^{p_2}(\R^n)$ into $L^p(\R^n)$ for $1\leq p_1,p_2\leq \infty$ satisfying the H\"{o}lder relation $\frac{1}{p_1}+\frac{1}{p_2}=\frac{1}{p}$ and the corresponding weighted analogues.  The bilinear problem is often significantly more difficult than its linear counterpart. There are a few recent papers in this direction. We provide here  a brief survey of known results. 

In what follows, we shall always assume that the exponents $p_1,p_2$ and $p$ satisfy the above H\"{o}lder relation, unless mentioned otherwise. 

In dimension $n=1$, the bilinear Bochner-Riesz problem is fairly well understood in~\cite{Bern1}, see Theorem 4.1 for precise details. In~\cite{Bern1}, F. Bernicot et al. obtained several positive and negative results for the bilinear Bochner-Riesz operator for $\alpha>0$. Note that the case $\alpha=0$ corresponds to the bilinear analogue of the ball multiplier problem. In a huge  contrast with the previously mentioned ball multiplier result due to C. Fefferman, in~\cite{GLi}, L. Grafakos and X. Li proved that the bilinear ball multiplier operator is bounded from $L^{p_1}(\R)\times L^{p_2}(\R)$ into $L^p(\R)$ in the strict local $L^2$ range:  $2< p_1,p_2,p'<\infty.$ The problem remains open for exponents $p_1,p_2$ and $p$ lying outside the local $L^2$ range, i.e., exactly one of $p_1,p_2,p'$ is less than $2$. However,  the bilinear ball multiplier operator in $\R^n,~n\geq 2,$ outside the local $L^2$ range is shown to be unbounded  in~\cite{DG} by G. Diestel and L. Grafakos. 

As in the linear case, it is easily verified that $\alpha=n-\frac{1}{2}$ is the critical index for the bilinear Bochner-Riesz problem. For $\alpha>n-\frac{1}{2}$, the operator $\mathcal {B}^{\alpha}$ maps $L^{p_1}(\R^n)\times L^{p_2}(\R^n)$ into $L^p(\R^n)$ for all $1\leq p_1,p_2\leq \infty$ satisfying the H\"{o}lder relation, see~\cite{Bern1} for instance. There are several positive results known in the literature on the boundedness of the operator with $\alpha$ below the critical index. Since our concern in this paper is to investigate the boundedness of the bilinear Bochner-Riesz operator $\mathcal {B}^{n-\frac{1}{2}}$, we do not get into the precise statements of all the known results as they are highly technical to describe. In particular, the following results are known.  
\begin{itemize}
\item ~\cite{Bern1} $\mathcal {B}^{\alpha}$ is bounded from $L^{2}(\mathbb{R}^{n})\times L^{2}(\mathbb{R}^{n})$ to $L^{1}(\mathbb{R}^{n})~$ for all $n\geq 1$ and $\alpha>0$,  
\item \cite{Lee2} $\mathcal {B}^{\alpha}$ is bounded from $L^{p_1}(\mathbb{R}^{2})\times L^{p_2}(\mathbb{R}^{2})$ to $L^{p}(\mathbb{R}^{2})$ for all $2\leq p_1,p_2\leq 4$ and $\alpha>0$. 
\end{itemize}
\begin{remark}
We would like to remark that in~\cite{Bern1, Lee2}  the authors have proved many other results.  Here we have described only a few selected statements for our convenience. Therefore, we refer to~\cite{Bern1,Lee2,HM} for a detailed account of known results in this direction. Further, we would like to refer to ~\cite{DHe,Gra2, EJeong} for the study of boundedness properties of the maximal bilinear Bochner-Riesz operators.
\end{remark} 
 
The bilinear (or multi-linear in general) analogues of the classical Muckenhoupt weights were systematically developed in~\cite{Ler1} by A.K. Lerner et al.. Since then, several authors have studied the weighted estimates for important classes of bilinear operators with respect to these bilinear weights. In this direction, we address the question of weighted estimates for the operator $\mathcal {B}^{n-\frac{1}{2}}$. Our method of proof exploits the ideas presented in~\cite{XQ, Bern1}. We use the structure of bilinear weights, see~\cite{Ler1}, along with  powerful extrapolation theorem for multi-linear weights proved in~~\cite{KJS, KJHS} and \cite{Bas}. 

Let us briefly recall the notion of Hardy-Littlewod maximal function and weights in the bilinear setting.

\subsection{Bilinear maximal function and weights}\label{pre}
We shall discuss the results in this section in bilinear setting only for notational convenience. However, we remark that the corresponding results are known in the general multi-linear setting. 

 Let $f_{1},f_{2}\in L^{1}_{loc}(\mathbb{R}^{n})$ be locally integrable functions. The bilinear Hardy-Littlewood maximal function $\mathcal{M}(f_{1},f_2)$ is defined by 
$$\mathcal{M}(f_{1},f_{2})(x)=\sup_{r>0}\prod^{2}_{j=1}\fint_{B(x,r)} f_{j} ,$$
 where the supremum is taken over all balls centered at $x\in \R^n$. 
 
The un-weighted $L^p$ estimates for the operator $\mathcal M$ immediately follow from the corresponding estimates for the classical Hardy-Littlewood maximal operator and H\"{o}lder's inequality. In~\cite{Ler1}, A. K. Lerner et al.  gave a complete characterization of the class of weights for which the bilinear operator $\mathcal M$ is bounded. Let us recall the notion of bilinear weights.  

\begin{definition}(Definition 3.5, \cite{Ler1})  \\ 
Let $1\leq p_{1},p_{2}< \infty $ and $\frac{1}{p}=\frac{1}{p_{1}}+\frac{1}{p_{2}}$. Denote $\vec{P}=(p_{1},p_{2})$.  Given a pair of weight functions      
$\vec{\omega}=(\omega_{1},\omega_{2})$, set $$v_{\omega}=\prod^{2}_{j=1}\omega^{\frac{p}{p_{j}}}_{j}$$
We say that $\vec{\omega}$ satisfies the bilinear $A_{\vec{P}}$ condition and write $\vec{\omega} \in A_{\vec{P}}$ if $$[\vec{\omega}]_{A_{\vec{P}}}:=\sup_{B}\left(\fint_{B}v_{\omega}\right)^{\frac{1}{p}} \prod^{2}_{j=1}\left(\fint_{B}\omega^{1-p'_{j}}_{j}\right)^{\frac{1}{p'_{j}}}<\infty.$$
When $p_{j}=1$, the quantity $\left(\fint_{B}\omega^{1-p'_{j}}_{j}\right)^{\frac{1}{p'_{j}}}$ has the standard interpretation  as $(\inf_{B}\omega_{j})^{-1}$. 
\end{definition}
\begin{theorem}(Theorem 3.7, \cite{Ler1})\label{bmf}
For $1<p_1,p_2<\infty$, the operator $\mathcal M$ is bounded from  $L^{p_{1}}(\omega_{1})\times L^{p_{2}}(\omega_{2})\rightarrow L^{p}(v_{\omega})$ if and only if $\vec{\omega} \in A_{\vec{P}}.$
\end{theorem}
We also refer to a recent paper~\cite{Bas} for a different formulation of multilinear weights, where the following notion of weights is considered. 
\begin{definition}[Definition $2.1$,\cite{Bas} ]
	\label{wNier}
	Let $\vec{r}=(r_1,r_2)$, $\vec{P}=(p_1,p_2)$ with $r_1,r_2\in (0,\infty)$ and $p_1, p_2\in (0,\infty]$. Let $p$ be given by $\frac{1}{p}=\frac{1}{p_1}+\frac{1}{p_2}$. We say $(\vec{r},s)\leq \vec{P}$ if $\vec{r}\leq \vec{P}$ and $p\leq s$ where $s\in (0,\infty]$. Here $\vec{r}\leq \vec{P}$ means that $r_j\leq p_j$, $j=1,2$. For weights $w_1, w_2$ write $w=\prod_{j=1}^2w_j$. We say that $\vec{w}=(w_1,w_2) \in A_{\vec{P},(\vec{r},s)}$ if 
	\begin{equation*}
	[\vec{w}]_{\vec{P},(\vec{r},s)}:=\sup_{Q} \Big(\prod_{j=1}^2 \langle w^{-1}_j\rangle_{\frac{1}{\frac{1}{r_j}-\frac{1}{p_j}}, Q} \langle w\rangle_{\frac{1}{\frac{1}{p}-\frac{1}{s}}, Q}\Big) <\infty,
	\end{equation*}
	where the supremum in the above is taken over all cubes (with sides parallel to coordinate axes) in $\mathbb R^n$. Here we have used the notation  $\langle f\rangle_{p, Q}=\left(\frac{1}{|Q|}\int_Q|f|^p\right)^{\frac{1}{p}}, 0<p<\infty$ and $\langle f\rangle_{\infty, Q}=\it{ess}sup_{x\in Q}|f(x)|.$
\end{definition} 
\begin{remark}
	Note that the case $p_{j}=\infty$ is included in the definition above. For $p_{j}=\infty$, the norm is interpreted as $\Vert f_{j}\Vert_{L^{p_{j}}(\omega^{p_{j}}_{j})}=\Vert f_{j}\omega_{j}\Vert_{L^{\infty}(\mathbb{R}^{n})}$. The weight class $A_{\vec{P},(\vec{r},s)}$ is equivalent to the weight class $A_{\vec{P}}$ with $\vec{P}=(p_{1},p_{2})$, when $\vec{r}=(1,1)$ and $s=\infty$. We have $(\omega^{p_{1}}_{1},\omega^{p_{2}}_{2})\in A_{\vec{P}}$ if and only if $\vec{\omega}=(\omega_{1},\omega_{2})\in A_{\vec{P},((1,1),\infty)}$. 
\end{remark}
\subsection{Main theorem}\label{pf}
The main result of this paper is the following. 
\begin{theorem} \label{mainthm} The bilinear Bochner-Riesz operator $\mathcal {B}^{n-\frac{1}{2}}$ is bounded from  $L^{p_{1}}(\omega_{1})\times L^{p_{2}}(\omega_{2})\rightarrow L^{p}(v_{\omega})$ for all $\vec{\omega}\in A_{\vec{P}}$ with $1<p_{1}, p_{2} \leq \infty$ and $\frac{1}{p_{1}}+\frac{1}{p_{2}}=\frac{1}{p}$. 
\end{theorem}
In~\cite{KJS} K. Li et al. established the multi-linear analogue of the classical Rubio de Francia's extrapolation theorem. We state here the bilinear version of their extrapolation result as follows. 
\begin{theorem}(Corollary 1.5,\cite{KJS})\label{extra} 
Let $\mathcal{F}$ be a collection of  triplets $(f,f_1,f_2)$ of non-negative functions. Let $\vec{P}=(p_{1},p_2)$ be exponents with $1\leq p_{1},p_2<\infty$, such that given any $\vec{\omega}\in A_{\vec{P}},$ the inequality 
\begin{eqnarray*}\Vert f\Vert_{L^{p}(v_{\omega})}\lesssim \prod^{2}_{i=1}\Vert f_{i}\Vert_{L^{p_{i}}(\omega_{i})}
\end{eqnarray*}
holds for all $(f,f_{1},f_{2})\in\mathcal{F}$, where $\frac{1}{p}=\frac{1}{p_{1}}+\frac{1}{p_{2}}$. Then for all exponents $\vec{Q}=(q_{1},q_{2})$ with $1< q_1,q_2<\infty$, and for all weights $\vec{v}=(v_{1},v_{2})\in A_{\vec{Q}}$ the inequality  
\begin{eqnarray*} \Vert f\Vert_{L^{q}(v)}\lesssim \prod^{2}_{i=1}\Vert f_{i}\Vert_{L^{q_{i}}(v_{i})}
\end{eqnarray*}
 holds for all $(f,f_{1},f_{2})\in\mathcal{F}$, where $\frac{1}{q}=\frac{1}{q_{1}}+\frac{1}{q_{2}}$ and $v=\prod^{2}_{i=1}v^{\frac{q}{q_{i}}}_{i}$.
\end{theorem}  
The notation $A\lesssim B$ means that there is a constant $C$ such that $A\leq CB.$

 Recently, in \cite{Bas} B. Nieraeth extended the extrapolation result for multilinear weights using a different approach. This allows us to deduce boundedness at the end-points when $p_{j}=\infty$ for some $j$, see Theorem~4.1 in~\cite{Bas}.  Also, \cite{KJHS}  K. Li et. al. revisited the extrapolation theory to include the end-points $p_{j}=\infty$ for some $j$. Here we follow the notation as in ~\cite{Bas}.

In view of the discussion above on extrapolation results, it is enough to prove the main Theorem~\ref{mainthm} for a single triplet and all weights in the corresponding class of bilinear weights. Therefore, we shall prove the following theorem and refer to it as our main result henceforth. 
\begin{theorem} \label{mainthm1} The bilinear Bochner-Riesz operator $\mathcal {B}^{n-\frac{1}{2}}$ is bounded from  $L^{2}(\omega_{1})\times L^{2}(\omega_{2})\rightarrow L^{1}(v_{\omega})$ for all $\vec{\omega}\in A_{\vec{P}},$ where $\vec{P}=(2,2).$
\end{theorem}

\subsection{Organization of the paper} We develop the auxiliary results required to prove the main theorem in section~\ref{main2:sec}. Section~\ref{pfmain} is devoted to proving the main theorem. 
\section{Auxiliary results} \label{main2:sec}

The following lemma is a partial substitute of the reverse H\"{o}lder inequality in the bilinear setting. 
\begin{lemma}\label{prop:wei}
Let $\vec{\omega}=(\omega_{1},\omega_{2})\in A_{\vec{P}}$, where $\frac{1}{p}=\frac{1}{p_{1}}+\frac{1}{p_{2}}$ with $1<p_1,p_2 <\infty$, then there exists a $\delta>0$, such that $\vec{\omega}_{\delta}=(\omega^{1+\delta}_{1},\omega^{1+\delta}_{2})\in A_{\vec{P}}$. 

\end{lemma}
Note that the above lemma is similar to Lemma 6.1 in \cite{Ler1} and the proof follows with no difficulty. However, for the sake of completeness, we provide the proof here. 
\begin{proof}
Using the characterization of bilinear weights in terms of the classical Muckenhoupt weights from~ \cite{Ler1}, we know that $\vec{\omega}\in A_{\vec{P}}$ if and only if $\omega^{1-p'_{j}}_{j}\in A_{2p'_{j}}$ for $j=1,2$ and $v_{\omega}=\prod^{2}_{j=1}\omega^{\frac{p}{p_{j}}}_{j}\in A_{2p}$.\\ 
The reverse H\"{o}lder inequality for $A_p$ weights yields that there exist $t_{j}>1$ and $C_{j}>0$ for all $j=1,2,3$, such that $$\left(\fint_{B}\omega^{t_{j}(1-p'_{j})}_{j}\right)^{\frac{1}{t_{j}}}\leq C_{j}\fint_{B}\omega^{1-p'_{j}}_{j}~~~\ \ \ \text{and}~~~~\ \ \ \ \  \left(\fint_{B}v^{t_{3}}_{\omega}\right)^{\frac{1}{t_{3}}}\leq C_{3}\fint_{B}v_{\omega}.$$

Set $t=\min\lbrace t_{1},t_2,t_{3}\rbrace$ and $C=\max\lbrace C_{1},C_2,C_{3}\rbrace$ and note that  

\begin{eqnarray*} \left(\fint_{B}\omega^{t(1-p'_{j})}_{j}\right)^{\frac{1}{t}}\leq C\fint_{B}\omega^{1-p'_{j}}_{j}
~~~\ \ \ \text{and}~~~~\ \ \ \ \ \left(\fint_{B}v^{t}_{\omega}\right)^{\frac{1}{t}}\leq C\fint_{B}v_{\omega}.
\end{eqnarray*}
Therefore, we have, 
\begin{eqnarray*}[\vec{\omega}_{t}]_{A_{\vec{P}}}
&=&\sup_{B\subset \mathbb{R}^{n}}\left(\fint_{B}v^{t}_{\omega}\right)^{\frac{1}{p}}\prod^{2}_{j=1}\left(\fint_{B}\omega^{t(1-p'_{j})}_{j}\right)^{\frac{1}{p'_{j}}}\\
&\leq & C \sup_{B\subset \mathbb{R}^{n}}\left(\fint_{B}v_{\omega}\right)^{\frac{t}{p}}\prod^{2}_{j=1}\left(\fint_{B}\omega^{(1-p'_{j})}_{j}\right)^{\frac{t}{p'_{j}}}\\
&=& C [\vec{\omega}]^{t}_{A_{\vec{P}}}.
\end{eqnarray*}
Finally, choose $\delta>0$ such that $t=1+\delta$ to complete the proof. 
\end{proof}

%


Note that one can consider the exponent $\alpha$ in the definition of $\mathcal B^{\alpha}$ to be a complex number. 

Next, we show that the operator $\mathcal {B}^{z}$ and its derivative with respect to the parameter $z$ satisfy the required $L^p$ estimates for certain index $z$ with $Re(z)< n-\frac{1}{2}$. These estimates play a key role in order to apply the analytic interpolation theorem proved in~\cite{LM} by L. Grafakos and M. Mastylo. We use the notation $\partial_x=\frac{\partial}{\partial x}$ and $\partial^{\alpha}_x=\frac{\partial^{\alpha}}{\partial x^{\alpha}}.$
\begin{lemma}\label{keylem1}
Let $n\geq 1$ and $z=Re(z)+\iota Im(z)$ be a complex number such that $\alpha_n< Re(z)< n-\frac{1}{2}$, where $\alpha_1=0$ and $\alpha_n=1$ for $n\geq 2$.  Then we have the following
\begin{eqnarray*} \int_{\mathbb{R}^n} |\left(\partial_z\right)^k \mathcal {B}^{z}(f,g)(x)| dx &\leq & C_{n+Re(z)} e^{\mathfrak{C} \vert Im(z)\vert^{2}}\Vert f\Vert_{L^{2}}\Vert g\Vert_{L^{2}}, \ \ \ \ ~\text{for}~\  k=0,1,
\end{eqnarray*}
where the constant $\mathfrak{C}>0$ comes from the asymptotic expansion of the Bessel function.
\end{lemma}
\begin{remark} The case $k=0$ in the above lemma is already known. See~\cite{Bern1, Lee2} for details. We exploit the ideas given in~\cite{FP,Bern1} to prove the other case  $k=1$. 
\end{remark}
\noindent {\bf Proof of Lemma~\ref{keylem1}}~ We need to prove the lemma for $k=1$ only. We shall consider the cases $n=1$ and $n\geq 2$ separately. \\
Note that for $k=1$, the multiplier symbol of the bilinear operator under consideration is given by 
$$m(\xi,\eta)=m_{0}(|\xi|, |\eta|)=(1- |\xi|^{2}-|\eta|^{2})^{z}_{+}\log(1- |\xi|^{2}-|\eta|^{2})_{+},~(\xi,\eta)\in \R^n\times \R^n.$$

\noindent {\bf Case I: $n=1$.} We follow the ideas from~\cite{FP} (see Proposition 6.1) to deal with this case. 

We perform the standard spherical smooth decomposition of the symbol, i.e., we write 

$$m(\xi,\eta)=m_{\phi}(\xi,\eta)+\sum\limits_{j=0}^{\infty} (j+3) 2^{-z(j+3)} m_{2^{-j}}(\xi,\eta),$$
where  
$$m_{2^{-j}}(\xi,\eta)=(j+3)^{-1} 2^{z(j+3)}(1- |\xi|^{2}-|\eta|^{2})^{z}_{+}\log(1- |\xi|^{2}-|\eta|^{2})_{+}\psi\left(\frac{1-|(\xi,\eta)|}{2^{-j}}\right)$$ and  $m_{\phi}(\xi,\eta)=m_{0}(|\xi|,|\eta|)\phi(|(\xi,\eta)|)$ with  $z=Re(z)+\iota Im(z)$ , $Re(z)>0$. Denote $ \psi_{2^{-j}}(t)=\psi\left(\frac{1-t}{2^{-j}}\right)$. And the functions  $\phi,\psi$ are two radial smooth functions supported in $B(0,\frac{1}{2})$ and annulus $Ann(\frac{1}{8},\frac{5}{8})=\{(\xi,\eta):\frac{1}{8}\leq |(\xi,\eta)|<\frac{5}{8}\}$ respectively such that $$\phi(t)+\sum_{j\geq0}\psi(\frac{1-t}{2^{-j}})=1.$$
We also assume that $\sup_{x\in\R}|\frac{d^k}{dx^k}\psi(x)|\leq c$ for all $0\leq k\leq 4,$ where $c$ is a fixed constant.

 Since the symbol $m_{\phi}$ is smooth so the corresponding kernel is integrable and the associated bilinear multiplier operator is bounded from $L^{2}(\omega_{1})\times L^{2}(\omega_{2})\rightarrow L^{1}(v_{\omega})$. Observe that the  function  $\psi_{2^{-j}}$ is smooth and supported in an annular region of inner and outer radius $1-5\cdot2^{-j-3}$ and $1-2^{-j-3}$ respectively ( see~\cite{Grafakosmodern}~and~\cite{Bern1}, Theorem $4.1$ for more details). 
We follow the strategy of the proof of Theorem $4.1$ from~\cite{Bern1} and show that the symbols $m_{2^{-j}}$ satisfy the hypothesis of Proposition 6.1 in~\cite{FP}. For an easy reference, we use the same notation as in ~\cite{FP} and denote $m_{2^{-j}}$ by $m_{\epsilon}$ with $\epsilon=2^{-j}$. 

Let $\nu(x)$ denote the distance of a point $x\in \R^2$ from the unit circle $\Gamma$. Let $\nabla\nu$ denote the direction of the local normal coordinate and $(\nabla\nu)^{\perp}$ denote the direction of tangential coordinate. We need to show that the symbol $m_{\epsilon}$  belongs to the class $\mathcal{N}^{\Gamma}_{\epsilon}$, which amounts to show that 
\begin{eqnarray}\label{prop61}
\left| \partial^{\alpha}_{\nabla \nu} \partial^{\beta}_{(\nabla \nu)^{\perp}} m_{\epsilon}\right|
&\lesssim & \epsilon^{-\vert \alpha\vert}
\end{eqnarray}
for sufficiently many multi-index $\alpha$ and $\beta$. 

See Definition 1.3 in~\cite{FP} for the exact definition of  
$\mathcal{N}^{\Gamma}_{\epsilon}$ and more details about the class. 

Let $x=(\xi,\eta)$ be in the open unit ball $B(0,1)$. Then $\nu(x)=1-\sqrt{(\xi^{2}+\eta^{2})}$ and 
$$\nabla\nu(x)=\left(\frac{-\xi}{\sqrt{(\xi^{2}+\eta^{2})}},\frac{-\eta}{\sqrt{(\xi^{2}+\eta^{2})}}\right)~\text{and}~(\nabla\nu)^{\perp}(x)=\left(\frac{\eta}{\sqrt{(\xi^{2}+\eta^{2})}},\frac{-\xi}{\sqrt{(\xi^{2}+\eta^{2})}}\right).$$
Note that  $\nabla m_{\epsilon}=(\partial_{\xi} m_{\epsilon}, \partial_{\eta} m_{\epsilon})$, 
where 
\begin{eqnarray*} \partial_{\xi} m_{\epsilon}
	&=&(j+3)^{-1}2^{z(j+3)}\left(-2\xi(1-\xi^{2}-\eta^{2})^{z-1}_{+}\left(1+z\log(1-\xi^{2}-\eta^{2})_{+}\right)\psi_{\epsilon}\right)\\
	& &+(j+3)^{-1}2^{z(j+3)}(1-\xi^{2}-\eta^{2})^{z}_{+}\log(1-\xi^{2}-\eta^{2})_{+}\partial_{\xi} \psi_{\epsilon}
\end{eqnarray*} 
and 
\begin{eqnarray*} \partial_{\eta} m_{\epsilon}
	&=&(j+3)^{-1}2^{z(j+3)}\left(-2\eta(1-\xi^{2}-\eta^{2})^{z-1}_{+}\left(1+z\log(1-\xi^{2}-\eta^{2})_{+}\right)\psi_{\epsilon}\right)\\
	&&+(j+3)^{-1}2^{z(j+3)}(1-\xi^{2}-\eta^{2})^{z}_{+}\log(1-\xi^{2}-\eta^{2})_{+}\partial_{\eta} \psi_{\epsilon}.
\end{eqnarray*} 

Therefore, $$\partial_{(\nabla \nu)^{\perp}} m_{\epsilon}=\nabla m_{\epsilon}\cdot (\nabla\nu)^{\perp}=(j+3)^{-1}2^{z(j+3)}\frac{(1-\xi^{2}-\eta^{2})^{z}_{+}}{\sqrt{(\xi^{2}+\eta^{2})}}\log(1-\xi^{2}-\eta^{2})_{+}\left(\eta \partial_{\xi} \psi_{\epsilon}-\xi \partial_{\eta} \psi_{\epsilon}\right)=0,$$
because $\left(\eta \partial_{\xi} \psi_{\epsilon}-\xi \partial_{\eta} \psi_{\epsilon}\right)=0$. Therefore, we only need to consider $\beta=0$. Now,
\begin{eqnarray*}
\partial_{\nabla\nu} m_{\epsilon}&=&\nabla m_{\epsilon}\cdot\nabla\nu\\
	&=&(j+3)^{-1}2^{z(j+3)}\left(2\sqrt{\xi^{2}+\eta^{2}}(1+z\log(1-\xi^{2}-\eta^{2}))(1-\xi^{2}-\eta^{2})^{z-1}_{+}\psi_{\epsilon}\right)\\ &+&(j+3)^{-1}2^{z(j+3)}\frac{1}{\epsilon}(1-\xi^{2}-\eta^{2})^{z}_{+}\log(1-\xi^{2}-\eta^{2})\psi'_{\epsilon}.
\end{eqnarray*} 
Therefore, 
\begin{eqnarray}
|\nabla m_{\epsilon}\cdot\nabla\nu|\lesssim \epsilon^{-1}.
\end{eqnarray} 
In the above we have used the fact that for  $Re(z)>0$, $(1-\xi^{2}-\eta^{2})^{z}_{+}\log(1-\xi^{2}-\eta^{2})_{+}$ is a bounded function and it is bounded by $(j+3)2^{-Re(z)(j+3)}$ \\
Now set  $\tilde{m_{\epsilon}}=\nabla m_{\epsilon}\cdot\nabla\nu$. Observe that the second term of $\nabla m_{\epsilon}\cdot\nabla\nu$ is similar to $\epsilon^{-1}m_{\epsilon}$. Therefore, it suffices to consider $\tilde{m}_{\epsilon}$ as\\  $(j+3)^{-1}2^{z(j+3)}\left(2\sqrt{\xi^{2}+\eta^{2}}(1+z\log(1-\xi^{2}-\eta^{2}))(1-\xi^{2}-\eta^{2})^{z-1}_{+}\psi_{\epsilon}\right)$.\\  Now,
\begin{eqnarray*}
	\partial^{2}_{\nabla\nu} {m}_{\epsilon}&=&\nabla\tilde{m}_{\epsilon}\cdot \nabla\nu,
\end{eqnarray*}
where 
\begin{eqnarray*}
	\nabla\tilde{m}_{\epsilon}&=&(\partial_{\xi} \tilde{m}_{\epsilon},
	\partial_{\eta} \tilde{m}_{\epsilon})~~~~~\text{and}\\
	\partial_{\xi} \tilde{m}_{\epsilon}&=&(j+3)^{-1}2^{z(j+3)}\frac{2\xi}{\sqrt{\xi^{2}+\eta^{2}}}(1-\xi^{2}-\eta^{2})_{+}^{z-1}(1+z\log(1-\xi^{2}-\eta^{2}))\psi_{\epsilon}\\ &+&(j+3)^{-1}2^{z(j+3)}2(z-1)(-2\xi)\sqrt{\xi^{2}+\eta^{2}}(1-\xi^{2}-\eta^{2})^{z-2}(1+z\log(1-\xi^{2}-\eta^{2}))\psi_{\epsilon}\\
	&+&(j+3)^{-1}2^{z(j+3)}2z(-2\xi)\sqrt{\xi^{2}+\eta^{2}} (1-\xi^{2}-\eta^{2})^{z-2}\psi_{\epsilon}\\
	&+&(j+3)^{-1}2^{z(j+3)}\frac{(-2\xi)}{\epsilon}(1-\xi^{2}-\eta^{2})^{z-1}(1+z\log(1-\xi^{2}-\eta^{2}))\psi'_{\epsilon}.
\end{eqnarray*}
Similarly we get,
\begin{eqnarray*}
	\partial_{\eta} \tilde{m}_{\epsilon}&=&(j+3)^{-1}2^{z(j+3)}\frac{2\eta}{\sqrt{\xi^{2}+\eta^{2}}}(1-\xi^{2}-\eta^{2})_{+}^{z-1}(1+z\log(1-\xi^{2}-\eta^{2}))\psi_{\epsilon}\\ &+&(j+3)^{-1}2^{z(j+3)}2(z-1)(-2\eta)\sqrt{\xi^{2}+\eta^{2}}(1-\xi^{2}-\eta^{2})^{z-2}(1+z\log(1-\xi^{2}-\eta^{2}))\psi_{\epsilon}\\
	&+&(j+3)^{-1}2^{z(j+3)}2z(-2\eta)\sqrt{\xi^{2}+\eta^{2}} (1-\xi^{2}-\eta^{2})^{z-2}\psi_{\epsilon}\\
	&+&(j+3)^{-1}2^{z(j+3)}\frac{(-2\eta)}{\epsilon}(1-\xi^{2}-\eta^{2})^{z-1}(1+z\log(1-\xi^{2}-\eta^{2}))\psi'_{\epsilon}.
\end{eqnarray*}
Now using the previous arguments we get 
\begin{eqnarray*}
	|\partial^{2}_{\nabla\nu}m_{\epsilon}|\leq C\epsilon^{-2}.
\end{eqnarray*} 

Similar arguments work for general multi-index $\alpha$ and $\beta$ and we get 

\begin{eqnarray*} \left| \partial^{\alpha}_{\nabla \nu}\partial^{\beta}_{(\nabla \nu)^{\perp}} m_{\epsilon}\right|\leq C \epsilon^{-\vert \alpha\vert},
\end{eqnarray*}  
where $C$ is a positive constant.

\noindent {\bf Case II: $n\geq 2$.} In this case $\alpha_n=1$. The desired result follows by invoking the following lemma from~\cite{Bern1}. 
\begin{lemma}(Lemma 3.7, \cite{Bern1})
Let $n\geq 2$ and $m_{0}$ be a bounded function supported in $[-1,1]^{2}$ such that 
\begin{eqnarray*}\partial_{\lambda_{1}}\partial_{\lambda_{2}}m_{0}(\lambda_{1}, \lambda_{2})\in L^{1}(\mathbb{R}^{2}).
\end{eqnarray*}
Define $m(\xi, \eta)=m_{0}(\vert \xi\vert, \vert \eta\vert)$ for $(\xi,\eta)\in \R^n \times \R^n$. Then the bilinear multiplier operator $T_{m}$ associated with $m$ is bounded from $L^{2}(\mathbb{R}^{n})\times L^{2}(\mathbb{R}^{n})$ into $L^{1}(\mathbb{R}^{n})$. Moreover, 
$$\Vert T_{m}\Vert_{L^{2}(\R^n)\times L^{2}(\R^n)\rightarrow L^{1}(\R^n)}\lesssim \|\partial_{\lambda_{1}}\partial_{\lambda_{2}}m_{0}\|_{L^1(\R^2)}.$$
 
\end{lemma}

We shall verify the hypothesis of the above lemma to deduce the desired result. Note that we have 
 $$\partial_{\lambda_{1}}\partial_{\lambda_{2}}m_{0}(\lambda_{1},\lambda_{2})=4z(z-1)\lambda_{1}\lambda_{2}(1- \lambda_{1}^{2}-\lambda_{2}^{2})^{z-2}_{+}\log(1- \lambda_{1}^{2}- \lambda_{2}^{2})_{+} + 4 \lambda_{1}\lambda_{2}(2z-1)(1- \lambda_{1}^{2}- \lambda_{2}^{2})^{z-2}_{+}.$$
 Therefore,
\begin{eqnarray*}
\int_{\mathbb{R}^{2}}\vert \partial_{\lambda_{1}}\partial_{\lambda_{2}}m_{0}(\lambda_{1},\lambda_{2})\vert d\lambda_{1}d \lambda_{2}
&\lesssim & \int_{B(0,1)}(1- \lambda_{1}^{2}-\lambda_{2}^{2})^{Re(z)-2}\vert\log(1- \lambda_{1}^{2}- \lambda_{2}^{2})\vert d\lambda_{1}d\lambda_{2}\\
& & +\int_{B(0,1)}(1- \lambda_{1}^{2}-\lambda_{2}^{2})^{Re(z)-2}d\lambda_{1}d \lambda_{2}\\
& \lesssim &  \int_{0}^{1}(1-u)^{Re(z)-2}du+ \int_{0}^{1}(1-u)^{Re(z)-2}\log(1-u)du \\
&  \lesssim & 1+ \int_{\frac{1}{2}}^{1}(1-u)^{Re(z)-2}du + \int_{\frac{1}{2}}^{1}(1-u)^{Re(z)-2}\log(1-u)du.
\end{eqnarray*}
Since $Re(z)>1,$ the above integrals are finite.
This completes the proof. 
\qed


\section{Weighted estimates for the bilinear Bochner-Riesz operator}\label{pfmain}

As mentioned previously our proofs are motivated from the ideas given in~\cite{XQ, Bern1}. The key ideas of the proof consist of the analytic interpolation for operators and the multi-linear extrapolation theorem. 

First, we shall observe that weighted estimates for the bilinear operator $\mathcal B^{z}$ hold when $Re(z)>n-\frac{1}{2}.$ Note that using the estimate on the kernel of the operator $\mathcal B^{z}$, the following   pointwise estimate holds 
\begin{eqnarray*}
|\mathcal B^{z}(f,g)(x)| &\lesssim & M(f)(x)M(g)(x),~
Re(z)>n-\frac{1}{2}.
\end{eqnarray*}
See~\cite{Bern1} for details. However, this does not yield weighted estimates for the operator $\mathcal B^{z}$ with respect to bilinear weights. We point out in the following lemma that pointwise domination of $\mathcal B^{z}(f,g)$ by the bilinear maximal function $\mathcal M(f,g)$ holds when $Re(z)>n-\frac{1}{2}$

\begin{lemma}
Let $n\geq 1$ and  $Re(z)>n-\frac{1}{2}$. Then the inequality 
$$|\mathcal {B}^{z}(f,g)(x)|\leq C_{n+Re(z)} e^{\mathfrak{C}\vert Im(z)\vert^{2}}\mathcal{M}(f,g)(x)$$ holds for all $f,g\in L^{1}_{loc}(\mathbb{R}^{n})$ and an absolute positive constant $C=C(n+Re(z))$ depending on the dimension $n$.

\end{lemma}
\begin{proof}
Let $z=\alpha+\iota t$, with $\alpha=n-\frac{1}{2}+\epsilon,~\epsilon>0$. We have 
$$\mathcal {B}^{z}(f,g)(x)=\int_{\R^n}\int_{\R^n} K_{z}(y_{1},y_{2})f(x-y_{1})g(x-y_{2})dy_{1}dy_{2},$$
where  $K_{z}(y) = \frac{\Gamma(\alpha+1+\iota t)}{\pi^{\alpha+\iota t}}\frac{J_{\alpha+n+\iota t}(2\pi \vert y\vert)}{\vert y\vert^{\alpha+n+\iota t}},~y=(y_{1},y_{2})$.

Invoking standard estimates of Bessel functions, the following kernel estimate holds.
$$\vert K_{z}(y_{1},y_{2})\vert\leq \frac{C(n+\alpha+\iota t)}{(1+ \vert y\vert)^{\alpha+n+\frac{1}{2}}},$$
where~ $C(n+\alpha+\iota t)\leq C_{n+\alpha}e^{\mathfrak{C}\vert Im(z)\vert^{2}}.$

Therefore, we have

\begin{eqnarray*}
\vert \mathcal {B}^{z}(f,g)(x)\vert
&\leq & C_{n+\alpha} e^{\mathfrak{C}\vert Im(z)\vert^{2}} \int_{\mathbb{R}^{n}}\int_{\mathbb{R}^{n}}\frac{\vert f(x-y_{1})\vert \vert g(x-y_{2})\vert}{(1+\vert y_{1}\vert+\vert y_{2}\vert)^{2n+\epsilon}}dy_{1}dy_{2}\\
&=& C_{n+\alpha} e^{\mathfrak{C}\vert Im (z)\vert^{2}}
\sum_{j=0}^{\infty}\sum_{k=0}^{\infty}\int_{2^jB \setminus 2^{j-1}B} \int_{2^kB \setminus 2^{k-1}B}\frac{\vert f(x-y_{1})\vert \vert g(x-y_{2})\vert}{(1+\vert y_{1}\vert+\vert y_{2}\vert)^{2n+\epsilon}}dy_{1}dy_{2},
\end{eqnarray*}
where $B$ denotes the unit ball in $\R^n$ and for $\lambda>0,$ the set $\lambda B$ denotes the dilated ball in $\R^n$ with center at the origin and radius $\lambda.$ With a slight abuse of notation, in the above expression, $2^kB \setminus 2^{k-1}B$ should be thought of as $B$ when $k=0.$ 
Because of symmetry, it is enough to estimate terms with $j\leq k$. Let us consider one such term.  

\begin{eqnarray*}
& & \int_{2^jB \setminus 2^{j-1}B} \int_{2^kB \setminus 2^{k-1}B}\frac{\vert f(x-y_{1})\vert \vert g(x-y_{2})\vert}{(1+\vert y_{1}\vert+\vert y_{2}\vert)^{2n+\epsilon}}dy_{1}dy_{2} \\
&\lesssim & \frac{1}{(2^k)^{2n+\epsilon}}\int_{\vert y_{1}\vert\leq 2^{k+1}}\int_{\vert y_{2}\vert\leq 2^{k+1}}\vert f(x-y_{1})\vert \vert g(x-y_{2})\vert dy_{1}dy_{2}\\
&\lesssim & 2^{-\frac{(k+j)\epsilon}{2}} \mathcal M(f,g)(x). 
\end{eqnarray*}
Therefore, one can sum up with respect to $k$ and $j$. 
This completes the proof. 
\end{proof}
The pointwise estimate in the above lemma along with the weighted estimates for the bilinear maximal function yield the following. 

\begin{corollary}\label{cor:weigh} Let $n\geq 1$ and $z\in \mathbb C$ be such that $Re(z)>n-\frac{1}{2}.$ Then the operator $\mathcal {B}^{z}$ is bounded from $L^{p_{1}}(\omega_{1})\times L^{p_{2}}(\omega_{2})\rightarrow L^{p}(v_{\omega})$ for all $\vec{\omega}\in A_{\vec{P}}$ with $1<p_{1}, p_{2}<\infty$ and $\frac{1}{p_{1}}+\frac{1}{p_{2}}=\frac{1}{p}$.
\end{corollary} 
As mentioned previously, we exploit the ideas from~\cite{XQ} to prove Theorem~\ref{mainthm1}. In the process, we need to prove certain estimates on the bilinear Bochner-Riesz operator $\mathcal B^{z}$ for some index $z$ with $\text{Re}(z)<n-\frac{1}{2}.$ 
\subsection{Proof of Theorem~\ref{mainthm1} }

We shall consider the  cases $n=1$ and $n\geq 2$ separately and provide a detailed proof for the latter case. The other case follows similarly.  

Let us consider the case when $n\geq 2$. 

For $\epsilon_1,\epsilon_2>0,$ and $N\in \mathbb N,$ consider the operator 
$$\tilde{\mathcal {B}}^{z,\epsilon_{1},\epsilon_{2},N}(f,g)(x)=\mathcal {B}^{(1+\epsilon_{1})(1-z)+z(n-\frac{1}{2}+\epsilon_{2})}(f,g)(x)(v_{N}(x))^{z}e^{A z^{2}},$$ 
such that $A>\mathfrak{C}$ and $v_{N}(x)$ is defined by

$$ v_{N}(x)= \begin{cases} v_{\omega}(x), & \mbox{if }  v_{\omega}(x)\leq N \\ N, & \mbox{if }v_{\omega}(x)>N. \end{cases}$$
Note that $v_{N}(x)\leq v_{\omega}(x)$, for a.e. $x$.
The parameters $\epsilon_1$ and $\epsilon_2$ will be chosen suitably at a later stage. 

Let $f,g$ be compactly supported positive smooth functions and $h$ be a bounded function. Further, let  $\delta_{0}>0$ be an arbitrarily small number and set 

$$f^{z}_{\delta_{0}}(x)=f(x)(\omega_{1}(x)+\delta_{0})^{-\frac{z}{2}}$$ and $$g^{z}_{\delta_{0}}(x)=g(x)(\omega_{2}(x)+\delta_{0})^{-\frac{z}{2}}$$

We use the duality for the operator under consideration and define the following function 
\begin{eqnarray*}
 \psi(z)&=&\int_{\mathbb{R}^{n}}\tilde{\mathcal {B}}^{z,\epsilon_{1},\epsilon_{2},N}(f^{z}_{\delta_{0}},g^{z}_{\delta_{0}})(x)h(x)dx \\ &=&\int_{\mathbb{R}^{n}}\mathcal {B}^{(1+\epsilon_{1})(1-z)+z(n-\frac{1}{2}+\epsilon_{2})}(f^{z}_{\delta_{0}},g^{z}_{\delta_{0}})(x)(v_{N}(x))^{z}h(x)e^{A z^{2}}dx,
\end{eqnarray*}
where  $0\leq Re(z)\leq 1$. 

Apply Lemma \ref{keylem1} to see that the function $\psi$ is analytic in the strip $S=\lbrace z\in \mathbb{C} : 0<Re(z)<1\rbrace$,  bounded and continuous on the closure $\bar{S}=\lbrace z\in \mathbb{C} : 0\leq Re(z)\leq 1\rbrace$. Consequently, the ``Three lines lemma" from complex analysis yields
\begin{eqnarray*}
\vert \psi(\theta)\vert
&\leq & C\left(\sup_{t\in \mathbb{R}}\vert \psi(\iota t)\vert\right)^{1-\theta}\left(\sup_{t\in \mathbb{R}}\vert \psi(1+\iota t)\vert\right)^{\theta},~~\theta\in (0,1).
\end{eqnarray*}

Note that, 
\begin{eqnarray*} \psi(i t)
&=& \int_{\R^n} \mathcal {B}^{(1+\epsilon_{1})(1-i t)+i t(n-\frac{1}{2}+\epsilon_{2})}\left(f(\omega_{1}+\delta_{0})^{\frac{-i t}{2}}, g(\omega_{2}+\delta_{0})^{\frac{-i t}{2}} \right)(x)(v_{N}(x))^{i t}h(x)e^{-A t^{2}}dx
\end{eqnarray*} 
Therefore, we have 
 \begin{eqnarray*}
  \sup_{t\in\mathbb{R}}\vert \psi(i t)\vert 
  &\leq &  \Vert h\Vert_{L^{\infty}}\sup_{t\in\mathbb{R}}e^{-At^{2}}\int_{\R^n} \vert \mathcal {B}^{(1+\epsilon_{1})(1-i t)+i t(n-\frac{1}{2}+\epsilon_{2})}\left(f(\omega_{1}+\delta_{0})^{\frac{-i t}{2}}, g(\omega_{2}+\delta_{0})^{\frac{-i t}{2}} \right)(x)\vert
  \end{eqnarray*}
Next observe that, 
$$\text{Re}[(1+\epsilon_{1})(1-i t)+i t(n-\frac{1}{2}+\epsilon_{2})]=1+\epsilon_{1}>1.$$
By Lemma~\ref{keylem1} we get the following
\begin{eqnarray*}
\sup_{t\in\mathbb{R}}\vert \psi(i t)\vert &\leq &
 C\Vert h\Vert_{L^{\infty}}\sup_{t\in\mathbb{R}}e^{-(A-\mathfrak{C})t^{2}}\left(\int_{\R^n} \vert f(\omega_{1}+\delta_{0})^{-\frac{i t}{2}}\vert^{2} dx \right)^{\frac{1}{2}}\left(\int_{\R^n} \vert g(\omega_{2}+\delta_{0})^{-\frac{i t}{2}}\vert^{2} dx \right)^{\frac{1}{2}}\\
&\leq & C_{\epsilon_{1},\epsilon_{2}}\Vert h\Vert_{L^{\infty}(\mathbb{R}^{n})}\Vert f\Vert_{L^{2}(\mathbb{R}^{n})}\Vert g\Vert_{L^{2}(\mathbb{R}^{n})}.
\end{eqnarray*}
In the above inequality, we have used that $A>\mathfrak{C}.$
  
In a similar fashion, we get the following estimate on $\psi(z)$ for $z=1+it$. 
\begin{eqnarray*}
& & \sup_{t\in\mathbb{R}} \vert \psi(1+i t)\vert\\
& \leq & \Vert h\Vert_{L^{\infty}(\mathbb{R}^{n})}\sup_{t\in \mathbb{R}}e^{A(1-t^{2})}\int_{\R^n} \vert \mathcal {B}^{(1+\epsilon_{1})(-i t)+(1+i t)(n-\frac{1}{2}+\epsilon_{2})}\left(f(\omega_{1}+\delta_{0})^{-\frac{1+i t}{2}}, g(\omega_{2}+\delta_{0})^{-\frac{1+i t}{2}} \right)(x)\vert v_{N}(x)dx\\
& \leq & \Vert h\Vert_{L^{\infty}(\mathbb{R}^{n})}\sup_{t\in \mathbb{R}}e^{A(1-t^{2})}\int_{\R^n} \vert \mathcal {B}^{(1+\epsilon_{1})(-i t)+(1+i t)(n-\frac{1}{2}+\epsilon_{2})}\left(f(\omega_{1}+\delta_{0})^{-\frac{1+i t}{2}}, g(\omega_{2}+\delta_{0})^{-\frac{1+i t}{2}} \right)(x)\vert v_{\omega}(x)dx.
\end{eqnarray*} 
Above we have used the fact that $v_{N}(x)\leq v_{\omega}(x)$, for a.e. $x$. 
In this case, 
$$\text{Re}[(1+\epsilon_{1})(-i t)+(1+i t)(n-\frac{1}{2}+\epsilon_{2})]=n-\frac{1}{2}+\epsilon_{2}>n-\frac{1}{2}.$$
We invoke Corollary~\ref{cor:weigh} to get the following estimate
 \begin{eqnarray*}
 & & \sup_{t\in\mathbb{R}}\vert \psi(1+i t)\vert  \\
 &\leq & C_{\epsilon_{1},\epsilon_{2}}\Vert h\Vert_{L^{\infty}}\sup_{t\in\mathbb{R}}e^{-(A-\mathfrak C)t^{2}}\left(\int_{\R^n} \vert f(\omega_{1}+\delta_{0})^{-\frac{1+i t}{2}}\vert^{2}\omega_{1}(x) dx \right)^{\frac{1}{2}}\left(\int_{\R^n} \vert g(\omega_{2}+\delta_{0})^{-\frac{1+i t}{2}}\vert^{2}\omega_{2}(x) dx \right)^{\frac{1}{2}}\\
& \leq & C_{\epsilon_{1},\epsilon_{2}}\Vert h\Vert_{L^{\infty}(\mathbb{R}^{n})}\Vert f\Vert_{L^{2}(\mathbb{R}^{n})}\Vert g\Vert_{L^{2}(\mathbb{R}^{n})}
\end{eqnarray*}
In the above we have used the fact that $(\omega_{j}+\delta_{0})^{-1}\leq \omega^{-1}_{j}$, for $j=0,1$. The three lines lemma along with the above estimates on $\psi$ at the boundary of the strip $S$ yields the following 
\begin{eqnarray}\label{three}
\vert\psi(\theta)\vert
&\leq & C_{\epsilon_{1},\epsilon_{2}}\Vert h\Vert_{L^{\infty}(\mathbb{R}^{n})}\Vert f\Vert_{L^{2}(\mathbb{R}^{n})}\Vert g\Vert_{L^{2}(\mathbb{R}^{n})},~~0<\theta<1.
\end{eqnarray} 
Since
\begin{eqnarray*}
\psi(\theta)
&=&\int_{\mathbb{R}^{n}}\mathcal {B}^{(1+\epsilon_{1})(1-\theta)+\theta(n-\frac{1}{2}+\epsilon_{2})}(f^{\theta}_{\delta_{0}},g^{\theta}_{\delta_{0}})(x)(v_{N}(x))^{\theta}h(x)dx \\
&=&\int_{\R^n} \mathcal {B}^{(1+\epsilon_{1})(1-\theta)+\theta(n-\frac{1}{2}+\epsilon_{2})}\left(f(\omega_{1}+\delta_{0})^{\frac{-\theta}{2}}, g(\omega_{2}+\delta_{0})^{\frac{-\theta}{2}} \right)(x)(v_{N}(x))^{\theta}h(x)dx
\end{eqnarray*}
The estimate~(\ref{three}) and standard duality argument give us that  
\begin{eqnarray*}
& & \int_{\R^n}  \vert\mathcal {B}^{(1+\epsilon_{1})(1-\theta)+\theta(n-\frac{1}{2}+\epsilon_{2})}\left(f(\omega_{1}+\delta_{0})^{\frac{-\theta}{2}}, g(\omega_{2}+\delta_{0})^{\frac{-\theta}{2}} \right)(x)\vert(v_{N}(x))^{\theta}dx\\
&\leq & C \left(\int_{\R^n} \vert f(x)\vert^{2} dx\right)^{\frac{1}{2}}\left(\int_{\R^n} \vert g(x)\vert^{2} dx\right)^{\frac{1}{2}}
\end{eqnarray*}
  
The constant $C$ in the above inequality is independent of $N$ and $\delta_{0}$ and hence letting $N\rightarrow \infty$, $\delta_{0}\rightarrow0$ and replacing $f$ and $g$ by  $f\omega_{1}^{\frac{\theta}{2}}$ and $g\omega_{2}^{\frac{\theta}{2}}$ respectively in the above inequality, we get 

\begin{align} \label{esti:wei}
\int_{\R^n} \vert\mathcal {B}^{(1+\epsilon_{1})(1-\theta)+\theta(n-\frac{1}{2}+\epsilon_{2})}\left(f, g \right)(x)\vert(v_{\omega}(x))^{\theta}dx
&\leq & C \left(\int_{\R^n} \vert f(x)\vert^{2}\omega_{1}^{\theta} dx\right)^{\frac{1}{2}}\left(\int_{\R^n} \vert g(x)\vert^{2}\omega_{2}^{\theta} dx\right)^{\frac{1}{2}},
\end{align}
where $0<\theta<1$ . 

Note that the above estimate~(\ref{esti:wei}) holds for all bilinear weights $\vec{\omega}\in A_{\vec{P}}$. 

Recall Lemma~\ref{prop:wei} which says that for any $\vec{\omega}\in A_{\vec{P}}$, there exists $\delta>0$ such that $\vec{\omega}_{\delta}=(\omega_{1}^{1+\delta},\omega_{2}^{1+\delta})\in A_{\vec{P}}$. 

For a given bilinear weight $\vec{\omega}\in A_{\vec{P}}$, the estimate ~(\ref{esti:wei}) holds true for the weight $\vec{\omega}_{\delta}\in A_{\vec{P}}$ with the choice $\theta=\frac{1}{1+\delta}$ and consequently we get the following

\begin{eqnarray}\label{critical}
\Vert \mathcal {B}^{\lambda}(f,g)\Vert_{L^{1}(v_{\omega})}
&\leq & C \Vert f\Vert_{L^{2}(\omega_{1})}\Vert g\Vert_{L^{2}(\omega_{2})},
\end{eqnarray} 
where $\lambda=(1+\epsilon_{1})(1-\frac{1}{1+\delta})+\frac{1}{1+\delta}(n-\frac{1}{2}+\epsilon_{2})$.

Finally, since $n\geq 2,$ we can choose $\epsilon_1$ and $\epsilon_2$ appropriately so that $\lambda=n-\frac{1}{2}.$ 
 
This completes the proof of Theorem~\ref{mainthm1} when $n\geq 2$. 
 
 As mentioned previously, the remaining case, i.e., the case $n=1$, may be completed similarly with the following modifications. 
 
 In this case, we consider the operator 
 
\begin{eqnarray*}
\tilde{\mathcal {B}}^{z,\epsilon_{1},\epsilon_{2},N}(f,g)(x)
&=&\mathcal {B}^{\epsilon_{1}(1-z)+z(n-\frac{1}{2}+\epsilon_{2})}(f,g)(x)(v_{N}(x))^{z}e^{A z^{2}}
\end{eqnarray*} 
We follow the argument as used in the  previous case along with Lemma~\ref{keylem1} for $n=1$. This will lead to estimate ~(\ref{critical}) with $\lambda=\epsilon_{1}(1-\frac{1}{1+\delta})+\frac{1}{1+\delta}(\frac{1}{2}+\epsilon_{2})$. Again, it is easy to see that one can choose 
$\epsilon_1$ and $\epsilon_2$ appropriately so that $\lambda=\frac{1}{2}.$

This completes the proof of Theorem~\ref{mainthm1} for $1<p_{1},p_{2}<\infty$. 

\section*{Acknowledgement} The authors would like to thank the referee for valuable suggestions. The second author was supported by Science and Engineering Research Board (SERB), Government of India, under the grant MATRICS: MTR/2017/000039/Math. The third  author is supported by CSIR (NET), file no. 09/1020 (0094)/2016-EMR-I. 

\bibliographystyle{plain}

\begin{thebibliography}{99}
\bibitem {FP} F. Bernicot; P. Germain, {\it Boundedness of bilinear multipliers whose symbols have a narrow support.}  J.Anal. Math. 119, 165--212 (2013).

\bibitem {Bern1} F. Bernicot; L. Grafakos; L. Song; L. Yan, {\it The bilinear Bochner-Riesz problem.} J. Anal. Math. 127, 179--217 (2015). 
 
\bibitem {B1} J. Bourgain,  {\it Besicovitch type maximal operators and applications to Fourier analysis.} Geom. Funct. Anal. 1, 147--187 (1991)

\bibitem{B2} J. Bourgain; L. Guth,  {\it Bounds on oscillatory integral operators based on multiplier estimates.} Geom. Funct. Anal. 21, 1239--1295 (2011).


\bibitem{MC} M. Christ, {\it Weak-type $(1,1)$ bounds for rough operators.} Ann. Math. 128 (1988), 19--42.

\bibitem{Di} J. M. Conde-Alonso; A. Culiuc; F. Di Plinio; Y. Ou, {\it A sparse domination principle for rough singular integrals.} Anal. PDE 10 (2017), no. 5, 1255--1284. 
\bibitem {F1} C. Fefferman, {\it Inequalities for strongly singular convolution operators.} Acta Math. 124, 9--36 (1970).

\bibitem {DG} G. Diestel; L. Grafakos, {\it  Unboundedness of the ball multiplier operator.} Nagoya  Math. J. 185, 151--159 (2007).

\bibitem {F2} C. Fefferman, {\it The multiplier problem for the ball.} Ann. Math.(2) 94, 330--336 (1971).

\bibitem{Grafakosmodern} L. Grafakos, {\it Modern Fourier analysis,} Third edition. Graduate Texts in Mathematics, 250. Springer, New York, 2014. 

\bibitem {GLi} L. Grafakos; X. Li, {\it  The disc as a bilinear multiplier.} Amer. J. Math. 128, 91--119 (2006).

\bibitem{Gra2} L. Grafakos; D. He; P. Hon\'{z}ik,  {\it Maximal operators associated with bilinear multipliers of limited decay.}  arXiv:1804.08527 [math.CA].

\bibitem {LM} L. Grafakos; M. Mastylo, {\it  Analytic families of multi-linear operators; Nonlinear Analysis.} 107 (2014).

\bibitem{DHe} D. He, {\it On the bilinear maximal Bochner-Riesz operators.} arXiv:1607.03527 [math.CA].

\bibitem{EJeong} E. Jeong; S. Lee, {\it Maximal estimates for the bilinear spherical averages and the bilinear Bochner-Riesz operators,} arXiv: 1903.07980v3, 14 Nov. 2019.

\bibitem {Lee2} E. Jeong; S. Lee; A. Vargas, {\it Improved bound for the bilinear Bochner-Riesz operator.} Math. Ann. 372 (2018), no. 1--2, 581--609.


\bibitem {Lee1} S. Lee, {\it Improved bounds for Bochner-Riesz and maximal Bochner-Riesz operators.} Duke Math. J. 122, 205--232 (2004).

\bibitem {Ler1} A. K. Lerner; S. Ombrosi; C. Perez; R.H. Torres; R. Trujillo-Gonzalez, New maximal functions and multiple weights for the multi-linear Calderon-Zygmund theory; Adv. Math. 220 (2009) 1222-1264.

\bibitem{KJHS} K. Li; J. M. Martell; H. Martikainen; S. Ombrosi; E. Vuorinen, {\it End-point estimates, extrapolation for multilinear Muckenhoupt classes, and applications,} arXiv: 1902.04951v1.

\bibitem {KJS} K. Li; J.M. Martell; S. Ombrosi, Extrapolation for multi-linear Muckenhoupt classes and applications to the bilinear Hilbert transform; arXiv: 1802.03338v2 [math.CA] 15th March, 2018.


\bibitem {HM} H. Liu; M. Wang, Boundedness of the bilinear Bochner-Riesz means in the non-Banach triangle case; arXiv: 1712.09235v1 [math.FA], 2017.  

\bibitem{Bas} B. Nieraeth, {\it Quantitative estimates and extrapolation for multilinear weight classes,} Math. Ann. 375 (2019), no. 1-2, 453-507.

\bibitem{Seeger} A. Seeger, Singular integral operators with rough convolution kernels; J. Amer. Math. Soc. 9 (1996), no. 1, 95--105. MR 1317232

\bibitem {XQ} X. Shi; Q. Sun, {\it Weighted norm inequalities for Bochner-Riesz operators and singular integral operators.}  Proc. Amer. Math. Soc., Vol. 116 (1992), no. 3, 665--673. MR 1136237

\bibitem{St1} E. M. Stein; G. Wiess, {\it Introduction to Fourier analysis on euclidean spaces.} Princeton University Press, 1971.  

\bibitem{tao} T. Tao, {\it Non-linear dispersive equations. Local and global analysis.} CBMS Regional Conference Series in Mathematics, 106, Amer. Math. Soc. Providence, RI, Washington, DC (2006).


\bibitem{AV} A. Vargas, {\it Weighted weak-type $(1,1)$ bounds for rough operators.} J. London Math. Soc. (2) 54 (1996), n. 2, 297--310. MR 1047758
\end{thebibliography}

\end{document}